\documentclass{SCAEOL}%SCAEOL for online version; SCAE for publication version; SCAES for the paper dedicated to somebody.
\numberwithin{equation}{section}
\numberwithin{table}{section}
\usepackage{amsfonts,amssymb,multirow}
\usepackage{color}
\usepackage{colortbl}
\begin{document}

%Basic Information
\Year{2019} %
\Month{June}
\Vol{60} %
\No{1} %
\BeginPage{1} %
\EndPage{XX} %
\AuthorMark{Xiaobin Liu {\it et al.}}
%\ReceivedDay{November 17, 2012}
%\AcceptedDay{January 22, 2013}
%\PublishedOnlineDay{; published online January 22, 2013}
%\DOI{10.1007/s11425-000-0000-0} % The author doesn't need fill in it.

% \title[short text for running head]{full title}{comments for title}
\title[Superconvergence of LDG methods]{Superconvergence of local discontinuous Galerkin methods with generalized alternating fluxes for 1D linear convection-diffusion equations}

% \author[]{Full name}{footnote}
% Remark:  One \author for one author

\author[1]{Xiaobin Liu}{}
\author[1]{Dazhi Zhang}{}
\author[2]{Xiong Meng}{Corresponding author}
\author[2]{Boying Wu}{}

\address[{\rm1}]{School of Mathematics, Harbin Institute of Technology, Harbin {\rm 150001}, Heilongjiang, China;}
\address[{\rm2}]{School of Mathematics and Institute for Advanced Study in Mathematics, \\Harbin Institute of Technology, Harbin {\rm 150001}, Heilongjiang, China;}
%\address[{\rm2}]{School of Mathematics and Institute for Advanced Study in Mathematics, Harbin Institute of Technology, Harbin {\rm 150001}, China;}
\Emails{liuxiaobin@hit.edu.cn,
zhangdazhi@hit.edu.cn, xiongmeng@hit.edu.cn, mathwby@hit.edu.cn}\maketitle

%     Abstract is required.

 {\begin{center}
\parbox{14.5cm}{\begin{abstract}
 This paper investigates superconvergence properties of the local discontinuous Galerkin methods with generalized alternating fluxes for one-dimensional linear convection-diffusion equations. By the technique of constructing some special correction functions, we prove the $(2k+1)$th order superconvergence for the cell averages, and the numerical traces in the discrete $L^2$ norm. In addition, superconvergence of order $k+2$ and $k+1$ are obtained for the error and its derivative at generalized Radau points. All theoretical findings are confirmed by numerical experiments.\vspace{-3mm}
\end{abstract}}\end{center}}

%  Keyword is required.

%  \subjclass is required.

%%%%%%%%%%%%%%%%%%%%%%%%%%%%%%%%%%%%%%%%%%%%%%%%%%%%%%%%%%%%
\renewcommand{\baselinestretch}{1.2}
\begin{center} \renewcommand{\arraystretch}{1.5}
{\begin{tabular}{lp{0.8\textwidth}} \hline \scriptsize
{\bf Citation:}\!\!\!\!&\scriptsize Xiaobin Liu, Dazhi Zhang, Xiong Meng, Boying Wu. Superconvergence of local discontinuous Galerkin methods with generalized alternating fluxes for 1D linear convection-diffusion equations. \vspace{1mm}
\\
\hline
\end{tabular}}\end{center}

%%%%%%%%%%%%%%%%%%%%%%%%%%%%%%%%%%%%%%%%%%%%%%%%%%%%%%%%%%%%
%% Text of article.
%%%%%%%%%%%%%%%%%%%%%%%%%%%%%%%%%%%%%%%%%%%%%%%%%%%%%%%%%%%%
%    Section headings
\baselineskip 11pt\parindent=10.8pt  \wuhao

\section{Introduction}
In this paper, we consider the local discontinuous Galerkin (LDG) methods for one-dimensional linear convection-diffusion equations
\begin{subequations}\label{1}
\begin{equation}
u_t+u_x-u_{xx}=0,\ \ \ \ \ \ \ (x,t)\in [0,2\pi]\times (0,T],
\end{equation}
\begin{equation}
u(x,0)=u_0(x),\ \ \ \ \ \ \ \ \ \ \ x\in R,~~~~~~~~~~~~~~~~~~~~
\end{equation}
\end{subequations}
where $u_0$ is sufficiently smooth. We will consider the periodic boundary condition $u(0,t)=u(2\pi,t)$, the mixed boundary condition $u(0,t)=g_1(t), u_x(2\pi,t)=g_2(t)$ and the Dirichlet boundary condition $u(0,t)=g_3(t),u(2\pi,t)=g_4(t)$. We study the superconvergence property concerning Radau points, cell averages, supercloseness of the LDG method with generalized alternating numerical fluxes, \textcolor{green}{including the case for which the parameters involved in the numerical fluxes for the prime variable regarding the convection part and the diffusion part are independently chosen for solving \eqref{1}.}

As an extension of discontinuous Galerkin (DG) method for solving first order hyperbolic equations, the LDG method was proposed by Cockburn and Shu \cite{Cockburn1} in the framework of solving second-order convection-diffusion equations. The idea of the LDG methods is to rewrite the original equation with high spatial derivatives as a first order system so that the DG method can be applied. Remark that in addition to stability issue, the local solvability of auxiliary variables introduced should also be guaranteed when choosing numerical fluxes.

Being a deeper insight of DG methods, superconvergence has been investigated basically measured in the discrete $L^2$ norm for Radau points as well \textcolor{green}{as} cell averages, in the $L^2$ norm for the error between the numerical solution and a particular projection of the exact solution (supercloseness), and in weak negative-order norm for enhancing accuracy. For example, by virtue of the duality argument in combination with the standard optimal a \emph{priori} error estimates,  Cockburn et al. \cite{Cockburn2} proved that the post-processed error is of order $2k+1$ superconvergence in the $L^2$ norm for linear hyperbolic systems and Ji et al. \cite{Ji} demonstrated that the smoothness-increasing accuracy-conserving (SIAC) filter can be extended to the multidimensional linear convection-diffusion equation in order to obtain $(2k+m)$th order superconvergence, where $m=0, \frac{1}{2}$ or $1$. Here and blow, $k$ denotes the polynomial degree of the discontinuous finite element space. Later, to efficiently compute multi-dimensional problems, the line filter, namely the one-dimensional kernel is designed via rotation in \cite{Julia}, and a rigorous proof of the post-processed errors is also given.  For arbitrary non-uniform regular meshes, by establishing the relation of the numerical solution and auxiliary variable as well as its time derivative, superconvergence of order $k+3/2$ is proved for linear convection-diffusion equations \cite{Cheng}. For supercloseness results concerning high order equations, see, for example, \cite{Hufford,Meng}. Note that aforementioned supercloseness results are not sharp. \textcolor{green}{In view of this, Yang and Shu \cite{Yang} adopted the dual argument to study the sharp superconvergence of the LDG method for one-dimensional linear parabolic equations, and improved superconvergence results of order $k+2$ were obtained in terms of supercloseness and Radau points.}

Recently, motivated by the successful applications of correction functions to finite element methods and finite volume methods for elliptic equations \cite{Chen}, Cao et al. \cite{Cao2,Cao3,Cao4,Cao7} studied superconvergence properties of DG and LDG methods for linear hyperbolic and parabolic equations. Specifically, they offered a novel proof to derive the $(2k+1)$th or $(2k+1/2)$th order superconvergence rate for the cell average and numerical fluxes, which will lead to the sharp $(k+2)$th order superconvergence for supercloseness as well as the function errors at downwind-biased points. Note that these superconvergent results are based on a supercloseness property of the DG solution to an interpolation function consisting of the difference between a standard Gauss--Radau projection of the exact solution and a carefully designed correction function. It is worth pointing out that a suitable correction is introduced to balance the difference between the projection error for the inner product term and for the DG operator term, and in standard optimal error estimates when a Gauss--Radau projection is used, the projection error involved in the DG operator term is exactly zero. This indicates that the standard  Gauss--Radau projection is not the best choice for superconvergence analysis. The superconvergence of the direct DG (DDG) method for the one-dimensional linear convection-diffusion equation was studied in \cite{Cao5}. We would like to remark that all the work mentioned above are focused on purely upwind and alternating numerical fluxes.

In order to obtain flexible numerical dissipation with potential applications to nonlinear systems, the upwind-biased flux was proposed in \cite{Meng3}, which is a linear combination of the numerical solution from both sides of interfaces. Stability and optimal error estimates were obtained by constructing and analyzing some suitable \emph{global} projections with emphasize on analysis of some circulant matrices. Note that the design of \emph{global} projections was similar to those in the work for the Burgers--Poisson equation \cite{Liu}. Moreover, Cheng et al. \cite{Cheng1} studied the LDG methods for the linear convection-diffusion equations when the generalized alternating fluxes were used, and they obtained the optimal $L^2$ norm error estimate in a unified setting, especially when numerical fluxes with different weights are considered. In \cite{Cao6}, Cao et al. investigated the superconvergence of DG methods based on upwind-biased fluxes for one-dimensional linear hyperbolic equations. More recently, Frean and Ryan \cite{Ryan1} proved that the use of SIAC filters was still able to extract the superconvergence information and obtain a globally smooth and superconvergent solution of order $2k+1$ for linear hyperbolic equations based on upwind-biased fluxes. Moreover, the $\alpha\beta$-fluxes, which were introduced as linear combinations of average and jumps of the solution as well as the auxiliary variables at cell interfaces, has been a hot research topic in recent years \cite{Ainsworth,Fu,Yingda}.

In current paper, we aim at analyzing the superconvergence properties of LDG methods using generalized alternating numerical fluxes for the convection-diffusion equations. The contribution of this paper is to consider the more flexible generalized alternating fluxes.
The critical step in deriving superconvergence is to construct special interpolation functions for both variables (the exact solution $u$ and the auxiliary variable $q$) with the aid of some suitable correction functions, essentially following \cite{Cao6}. Taking into account the stability result, we use special projections to eliminate or control the troublesome jump terms involving projection errors; see, e.g. Lemma \ref{ww}. To be more precise, we will establish the superconvergence between the LDG solution $(u_h,q_h)$ and special interpolation functions $u^\ell_I=P_{\theta}u-W^\ell_u$ as well as $q^\ell_I=P_{\tilde{\theta}}q-W^\ell_q$, where $W^\ell_u$ and $W^\ell_q$ are correction functions to be specified later, with the main technicality being the construction and analysis of some suitable projections tailored to the very choice of the numerical fluxes. By a rigorous mathematical proof, we prove a superconvergence rate of $2k+1$ for the errors of numerical traces and for the cell averages, and $k+2$ for the DG error at generalized Radau points.

The paper is organized as follows. In section \ref{ldg}, we present the LDG method with generalized alternating fluxes. In section \ref{correction}, we construct special functions to correct the error between the LDG solution and the standard Gauss--Radau projections of the exact solution. Section \ref{super} is the main body of the paper, in which we show and prove some superconvergence phenomena for cell averages and generalized Radau points for periodic boundary conditions. Other boundary cases including the mixed boundary condition and Dirichlet boundary condition will be considered in section \ref{extension}, and the choice of numerical initial discretization is also given. In section \ref{numerical}, we present some numerical experiments that confirm the sharpness of our theoretical results. We will end in section \ref{summary} with concluding remarks and some possible future work.

\section{The LDG scheme}\label{ldg}

In this section, we present the LDG scheme with generalized alternating fluxes for the linear convection-diffusion equation \eqref{1}. As usual, we divide the computational domain $\Omega=[0,2\pi]$ into $N$ cells
\begin{equation*}
0=x_{\frac{1}{2}}<x_{\frac{3}{2}}<\cdot\cdot\cdot<x_{N+\frac{1}{2}}=2\pi.
\end{equation*}
For any positive integer $r$, we define $\mathbb{Z}_r=\{1,\cdot\cdot\cdot,r\}$ and denote
\begin{align*}
x_j=\frac{1}{2}\big(x_{j-\frac{1}{2}}+x_{j+\frac{1}{2}}\big),\ \ \ \ I_j=\big(x_{j-\frac{1}{2}},x_{j+\frac{1}{2}}\big),\ \ \ j\in \mathbb{Z}_N
\end{align*}
as the cell centers and cells, respectively. Let $h_j=x_{j+\frac{1}{2}}-x_{j-\frac{1}{2}}$ be the length of the cell $I_j$ for $j\in \mathbb{Z}_N$, and $h=\max\limits_{1\leq j\leq N}h_j$. We assume the partition $\Omega_h$ is quasi-uniform in the sense that there exists a constant $C$ independent of $h$ such that $Ch\leq h_j\leq h$. Define the finite element space
\begin{align*}
V^k_h=\{v\in L^2(\Omega):v|_{I_j}\in P^k(I_j),\ \forall j\in\mathbb{Z}_N\},
\end{align*}
where $P^k(I_j)$ is the space of polynomials on $I_j$ of degree at most $k\geq 0$. We use
\begin{align*}
\bar{u}_{j+\frac{1}{2}}=\frac{1}{2}(u^+_{j+\frac{1}{2}}+u^-_{j+\frac{1}{2}}),\ \ \ \ \ \ \ [u]_{j+\frac{1}{2}}=u^+_{j+\frac{1}{2}}-u^-_{j+\frac{1}{2}}
\end{align*}
to denote the mean and jump of the function $u$ at each element boundary point $x_{j+\frac{1}{2}}$, and the weighted average is denoted by
\begin{align*}
u^{(\theta)}_{j+\frac{1}{2}}=\theta u^-_{j+\frac{1}{2}}+\tilde{\theta}u^+_{j+\frac{1}{2}}, ~~ \tilde{\theta}=1-\theta,
\end{align*}
where $u^+_{j+\frac{1}{2}}$ and $u^-_{j+\frac{1}{2}}$ are the traces from the right and left cells, respectively.

Throughout this paper, we employ $W^{\ell,p}(D)$  to denote the standard Sobolev space on $D$ equipped with the norm $\|\cdot\|_{W^{\ell,p}(D)}$ with $\ell\geq 0, p=2, \infty$. For simplicity,  we set $\|\cdot\|_{W^{\ell,p}(D)}=\|\cdot\|_{\ell, p, D}$ with $D= \Omega$ or $I_j$. The subscript $D$ will be omitted when $D= \Omega$, \textcolor{green}{and $W^{\ell,p}(D)$ can be written as $H^\ell(D)$ when $p=2$.}

In order to construct the LDG scheme, we first introduce an auxiliary variable $q=u_x$, then the problem \eqref{1} can be written into a first order system
\begin{equation}\label{13}
u_t + (u-q)_x =0, ~~~q - u_x = 0,
\end{equation}
where $(u-q,u)$ is the physical flux and $u$ is the so-called prime variable.
The LDG scheme is thus to find $u_h, q_h\in V^k_h$ such that for all test functions $v,\psi\in
V^k_h$

\begin{subequations}\label{2}
\begin{equation}
(u_{ht},v)_j-(u_h-q_h,v_x)_j+(\tilde{u}_h-\hat{q}_h)v^-|_{j+\frac{1}{2}}-(\tilde{u}_h-\hat{q}_h)v^+|_{j-\frac{1}{2}}=0,\\
\end{equation}
\begin{equation}
(q_h,\psi)_{j}+(u_h,\psi_x)_{j}-\hat{u}_h\psi^-|_{j+\frac{1}{2}}+\hat{u}_h\psi^+|_{j-\frac{1}{2}}=0. ~~~~~~~~~~~~~~~~~~~~~~~~
\end{equation}
\end{subequations}
Here $(u,v)_j=\int_{I_j}uvdx$, and $\tilde{u}_h, \hat{q}_h,\hat{u}_h$ are numerical fluxes. We use the generalized alternating numerical fluxes related to arbitrary parameters $\lambda$ and $\theta$ as in \cite{Cheng1}. That is,
\begin{align}\label{23}
(\tilde{u}_h-\hat{q}_h,\hat{u}_h)=(u^{(\lambda)}_h-q^{(\tilde{\theta})}_h,u^{(\theta)}_h).
\end{align}
Remark that the parameters in the numerical flux regarding the convection part and diffusion part can be chosen independently, and to ensure stability the weight $\lambda$ should satisfy $\lambda \ge \frac{1}{2}$.

For simplicity, we introduce the notation pertaining to the DG operator
\begin{align*}
\mathcal{H}^1(u,q;v)=\sum^N_{j=1}\mathcal{H}^1_j(u,q;v),\ \ \ \ \ \ \textcolor{green}{\mathcal{H}^2(u;\psi)=\sum^N_{j=1}\mathcal{H}^2_j(u;\psi),}
\end{align*}
where
\begin{align*}
\mathcal{H}^1_j(u,q;v)  &= (q-u,v_x)_j- (\hat{q} - \tilde u)v^-|_{j+\frac{1}{2}}+(\hat{q} - \tilde u)v^+|_{j-\frac{1}{2}},\\
\textcolor{green}{\mathcal{H}^2_j(u;\psi)} &   =  \textcolor{green}{  (u,\psi_x)_{j}-\hat{u}\psi^-|_{j+\frac{1}{2}}+\hat{u}\psi^+|_{j-\frac{1}{2}}.}
\end{align*}
Thus, by Galerkin orthogonality, the cell error equation can be written as
\begin{align}\label{25}
({e_u}_t,v)_j+(e_q,\psi)_{j}+\mathcal{H}^1_j(e_u,e_q;v)+\textcolor{green}{\mathcal{H}^2_j(e_u;\psi)}=0,\ \ \ \forall v,\psi \in V_h^k,
\end{align}
where $e_u=u-u_h, e_q=q-q_h$.

For optimal error estimates of the LDG scheme using the generalized numerical fluxes \eqref{23} solving convection-diffusion equations with periodic boundary conditions, a globally defined projection $P_\theta$ together with $P_{\tilde \theta}$ is usually needed. For $z\in H^1(\Omega_h) = \cup_{j\in \mathbb Z_N} H^1(I_j)$, the generalized Gauss--Radau projection $P_\theta z$ is defined as the element of $V^k_h$ that satisfies
\begin{subequations}
\begin{align}\label{18}
&\int_{I_j}(P_\theta z-z)v_hdx=0,\ \ \ \ \ \forall v_h\in P^{k-1}(I_j), \\
&\big(P_\theta z\big)^{(\theta)}_{j+\frac{1}{2}}=\big(z^{(\theta)}\big)_{j+\frac{1}{2}}, \ \ \ \ \ \ \forall j\in \mathbb{Z}_N.~~~~~~~~
\end{align}
\end{subequations}
It has been shown in \cite{Liu, Meng3} that the projection $P_\theta z$ is well defined for $\theta \neq \frac{1}{2}$, and for $\theta = 1/2$ some restrictions on the mesh as well as polynomial degree are needed to guarantee existence and optimal approximation property of the projection \cite{Bona}. Note that when the parameter $\theta$ is taken as $0$ or $1$, the projection $P_\theta$ reduces to the standard local Gauss--Radau projection $P^+_h$ or $P^-_h$ as defined in \cite{Castillo2}. Besides, the projection $P_\theta$ satisfies the following optimal approximation property \cite{Liu, Meng3}
\begin{align}\label{26}
\|z-P_\theta z\|_{I_j}+h^{\frac{1}{2}}\|z-P_\theta z\|_{\infty,I_j}\leq Ch^{k+\frac{3}{2}}\|z\|_{k+1,\infty},
\end{align}
where $C>0$ is independent of $h$ and $z$.

To obtain the superconvergence results, the following lemma is useful in describing correction functions.
\begin{lemma}\cite{Cao6}\label{27}
Suppose $A$ is an $N\times N$ circulant matrix with the first row $(\theta,(-1)^k(1-\theta),0,...,0)$ and the last row $((-1)^k(1-\theta),0,0,...,\theta)$, where $\theta > 1/2$. Then, for any vectors $X=(x_1,...,x_N)^T, b=(b_1,...,b_N)^T$ satisfying $AX=b$, there holds
\begin{equation*}
|x_j| \lesssim \| b\|_\infty, \quad \forall j \in \mathbb Z_N.
\end{equation*}
\end{lemma}

\section{Correction functions}\label{correction}
In what follows, we will present the construction of correction functions. The cases for the weights of the prime variable $u_h$ in \eqref{23} being the same or different are discussed in the following two subsections.
\textcolor{red}{\subsection{The case with $\lambda = \theta$ in \eqref{23}}\label{same}}
When $\lambda = \theta$ in \eqref{23}, to construct special interpolation functions $(u^{\ell}_I,q^{\ell}_I)$ by modifying generalized Gauss--Radau projections with correction functions so that they are superclose to the LDG solution $(u_h,q_h)$, we start by denoting $L_{j,k}$ as the standard Legendre polynomial of degree $k$ on the interval $I_j$, and assume that the function $v(x,t)$ has the following Legendre expansion. That is, on each $I_j, j\in \mathbb{Z}_N$,
\begin{align*}
v(x,t)=\sum^\infty_{m=0}v_{j,m}(t)L_{j,m}(x),\ \ \ \ \ \ v_{j,m}=\frac{2m+1}{h_j}(v,L_{j,m})_j.
\end{align*}
By the definition of $P_{\theta}$ in \eqref{18}, we can rewrite $P_\theta v$ into the following form
\begin{align*}
P_{\theta}v=\sum^k_{m=0}v_{j,m}(t)L_{j,m}(x)+\bar{v}_{j,k}(t)L_{j,k}(x),
\end{align*}
where $\bar{v}_{j,k}$ can be determined by $(v- P_\theta v)_{j+1/2}^{(\theta)} = 0$ with
\begin{align}
v-P_{\theta}v=-\bar{v}_{j,k}(t)L_{j,k}(x)+\sum^\infty_{m=k+1}v_{j,m}(t)L_{j,m}(x).
\end{align}
It follows from the orthogonality of Legendre polynomials and \eqref{26} that,
\begin{align*}
|\bar{v}_{j,k}|\lesssim \frac{2k+1}{h_j}|(v-P_\theta v,L_{j,k})_j|\lesssim h^{k+1}\|v\|_{k+1,\infty}.
\end{align*}
Following \cite{Cao6}, to balance projection errors for the inner product term and the DG operator term, we define an integral operator $D^{-1}_x$ by
$$
D^{-1}_xu(x)=\frac{1}{\bar{h}_j}\int^x_{x_{j-\frac{1}{2}}}u(\tau)d\tau,\ \ \ \ \ \tau\in I_j,
$$
where $\bar h_j = h_j/2$. Obviously, $u(x) = \bar h_j \big(D^{-1}_xu(x)\big)_x $. Moreover, by the properties of Legendre polynomials, we have
\begin{equation}\label{primal}
D^{-1}_xL_{j,k}(x)=\frac{1}{2k+1}(L_{j,k+1}-L_{j,k-1})(x).
\end{equation}

To clearly see how to cancel terms involving projection errors with the goal of obtaining superconvergence, we split the error $e_u,e_q$ into two parts:
\begin{align*}
&e_u=u-u_h=u-u^\ell_I+u^\ell_I-u_h \triangleq \eta_u+\xi_u,\\
&e_q=q-q_h=q-q^\ell_I+q^\ell_I-q_h\triangleq\eta_q+\xi_q.
\end{align*}
Then error equation (\ref{25}) becomes
\begin{align*}
&(\xi_{ut},v)_j+(\xi_q,\psi)_{j}+\mathcal{H}^1_j(\xi_u,\xi_q;v)+\textcolor{green}{\mathcal{H}^2_j(\xi_u;\psi)}\\
=&-(\eta_{ut},v)_j-(\eta_q,\psi)_{j}-\mathcal{H}^1_j(\eta_u,\eta_q;v)-\textcolor{green}{\mathcal{H}^2_j(\eta_u;\psi)}.
\end{align*}
For periodic boundary conditions, by choosing $v=\xi_u, \psi=\xi_q$ and summing over all $j$, we have
\begin{align}
&\frac{1}{2}\frac{d}{dt}\|\xi_u\|^2+\|\xi_q\|^2+(\lambda-\frac{1}{2})\sum^N_{j=1}[\xi_u]^2_{j+\frac{1}{2}}\nonumber \\
=&-(\eta_{ut},\xi_u)-(\eta_q,\xi_q)-\mathcal{H}^1(\eta_u,\eta_q;\xi_u)
-\textcolor{green}{\mathcal{H}^2(\eta_u;\xi_q)}.\label{24}
\end{align}
From the equation \eqref{24}, we can see that in order to obtain the supercloseness properties between the numerical solution $u_h$ and interpolation function $u^\ell_I$, we need to obtain a sharp superconvergent bound for the right-hand term, essentially using the switch of the time derivative and spatial derivative through the integral operator $D_x^{-1}$ in combination with integration by parts; see Lemma \ref{ww} below. Next, we show how to construct interpolation functions and estimate the right-hand side of \eqref{24}.

To construct the interpolation functions $(u^\ell_I,q^\ell_I)$, we define a series of functions $w_{u,i},w_{q,i}\in V_h^k, i\in\mathbb{Z}_k$ as follows:
\begin{subequations}\label{28}
\begin{align}
&(w_{u,i}-\bar{h}_jD_x^{-1}w_{q,i-1},v)_j=0,\ \ \ \ \ \ \ \ \ \ \ \ \ \ \big(w^{(\theta)}_{u,i}\big)_{{j+\frac{1}{2}}}=0, \label{5} \\
&(w_{q,i}-w_{u,i}-\bar{h}_jD_x^{-1}\partial_tw_{u,i-1},v)_j=0,\ \ \ \big(w^{(\tilde{\theta})}_{q,i}\big)_{{j+\frac{1}{2}}}=0, \label{4}
\end{align}
\end{subequations}
where $v\in P^{k-1}(I_j)$, and
\begin{align*}
w_{u,0}=u-P_\theta u,\ \ \ \ \ \ w_{q,0}=q-P_{\tilde{\theta}} q.
\end{align*}

\begin{lemma}\label{w}
The functions $w_{u,i},w_{q,i}, i\in \mathbb{Z}_k$ defined in (\ref{28}) have the following properties
\begin{subequations}
\begin{equation}
\|\partial_tw_{u,i}\|_{\infty}\lesssim h^{k+i+1}\|u\|_{k+i+3,\infty},\ \ \|w_{q,i}\|_{\infty}\lesssim h^{k+i+1}\|u\|_{k+i+2,\infty},\label{6}
\end{equation}
\begin{equation}\label{7}
(w_{u,i},v)_j=0,\ \qquad (w_{q,i},v)_j=0,\ \qquad \forall v\in P^{k-i-1}(I_j).
\end{equation}
\end{subequations}
\end{lemma}

\begin{proof}\quad
The proof of this lemma is based on deriving the following expression of $w_{u,i}$ and $w_{q,i}$ in each element $I_j$, which can be obtained by induction. It reads,
\begin{align}
w_{u,i}\big|_{I_j}=\sum^k_{m=k-i}\beta^j_{i,m}L_{j,m}(x),\ \ \ \ w_{q,i}\big|_{I_j}=\sum^k_{m=k-i}\gamma^j_{i,m}L_{j,m}(x),\ \ \ i\in \mathbb{Z}_k.\label{17}
\end{align}

\emph{Step 1}: When $i=1$, by taking $v=L_{j,m}$ with $m\leq k-1$ in \eqref{5} and using \eqref{primal} together with the orthogonality property of Legendre polynomials, we obtain
\begin{align*}
(w_{u,1}-\bar{h}_jD_x^{-1}w_{q,0},v)=(\beta^j_{1,k-1}L_{j,k-1}-\frac{\bar{q}_{j,k}}{2k+1}\bar{h}_jL_{j,k-1},v)=0.
\end{align*}
Obviously, $\beta^j_{1,k-1}=\frac{\bar{q}_{j,k}}{2k+1}\bar{h}_j$, where $\bar{q}_{j,k}$ is the coefficient of the Legendre expansion for $q$; see (3.1) with $v$ replaced by $q$ and $P_{\theta}$ replaced by $P_{\tilde \theta}$. Using the fact that $\big(w^{(\theta)}_{u,1}\big)_{j+\frac{1}{2}}=0$ we have
\begin{align}
\theta\beta^j_{1,k}+(-1)^k(1-\theta)\beta^{j+1}_{1,k}=(-1)^{k}(1-\theta)\beta^{j+1}_{1,k-1}-\theta\beta^{j}_{1,k-1}.\label{16}
\end{align}
Then the linear system (\ref{16}) can be written in the matrix-vector form
\begin{align*}
A\beta_{1,k}=b,
\end{align*}
where $A=circ(\theta,(-1)^k(1-\theta),0,...,0)$ is an $N\times N$ circulant matrix, and
\begin{equation*}
\beta_{1,k}=\begin{pmatrix}\beta^1_{1,k}\\\beta^2_{1,k}\\
\vdots\\\beta^N_{1,k}\end{pmatrix},\ \ \ \ \ b=\begin{pmatrix}-\theta\beta^1_{1,k-1}+(-1)^k(1-\theta)\beta^2_{1,k-1}
\\-\theta\beta^2_{1,k-1}+(-1)^k(1-\theta)\beta^3_{1,k-1}\\\vdots\\
-\theta\beta^N_{1,k-1}+(-1)^k(1-\theta)\beta^{1}_{1,k-1}
\end{pmatrix}.
\end{equation*}
It is easy to compute the determinant of the $A$ in the form
\begin{align*}
|A|=\theta^N(1-p^N),\ \ \ \ p=\frac{(-1)^k(\theta-1)}{\theta},
\end{align*}
 and for $\theta\neq \frac{1}{2}$ the matrix $A$ is always invertible. Therefore, the linear system (\ref{16}) has the unique solution. Moreover, by Lemma \ref{27}, we have
\begin{align*}
|\beta^j_{1,k}|\lesssim \max_{1\leqslant \ell\leqslant N}|b_\ell|\lesssim h^{k+2}\|u\|_{k+2,\infty},\ \ \forall j\in \mathbb{Z}_N.
\end{align*}
Thus,
\begin{align*}
\|\partial_tw_{u,1}\|_{\infty, I_j}=\|\partial_t(\beta^j_{1,k-1}L_{j,k-1}+\beta^j_{1,k}L_{j,k})
\|_{\infty,I_j}\lesssim h_j|\partial_t\bar{q}_{j,k}|\lesssim h^{k+2}\|u\|_{k+4,\infty}.
\end{align*}

Similarly, when choosing $v=L_{j,m}, m\leq k-1$ in (\ref{4}), we obtain
$$
w_{q,1}\big|_{I_j}=\sum^k_{m=k-1}\gamma^j_{1,m}L_{j,m},
$$
where
 $$\gamma^j_{1,k-1}=\beta^j_{1,k-1}
+\frac{\partial_t\bar{u}_{j,k}}{2k+1}\bar{h}_j,$$ and $\gamma^j_{1,k}$ is the solution of the following linear system
\begin{align*}
\tilde{A}\gamma_{1,k}=\tilde{b},
\end{align*}
with $\tilde{A}=circ(\tilde{\theta},(-1)^k(1-\tilde{\theta}),0,...,0)$ being an $N\times N$ circulant matrix, and
\begin{equation*}
\gamma_{1,k}=\begin{pmatrix}\gamma^1_{1,k}\\\gamma^2_{1,k}\\\vdots\\\gamma^N_{1,k}\end{pmatrix},\ \ \ \ \ \tilde{b}=\begin{pmatrix}-\tilde{\theta}\gamma^1_{1,k-1}+(-1)^k(1-\tilde{\theta})\gamma^2_{1,k-1}\\-\tilde{\theta}\gamma^2_{1,k-1}
+(-1)^k(1-\tilde{\theta})\gamma^3_{1,k-1}\\\vdots\\-\tilde{\theta}\gamma^N_{1,k-1}+(-1)^k(1-\tilde{\theta})\gamma^{1}_{1,k-1}
\end{pmatrix}.
\end{equation*}
Consequently, the estimate of $\|w_{q,1}\|_{\infty}$ in \eqref{6} follows by using Lemma \ref{27} and optimal approximation property \eqref{26}. Moreover, \eqref{7} is a trivial consequence of the expression \eqref{17} when the orthogonality property of Legendre polynomials is taken into account.

\emph{Step 2}: Suppose that (\ref{6}) and (\ref{17}) are valid for all $i\le k -1$ and we want to prove it still holds for $i+1$. From equation \eqref{5}, we can get
\begin{align*}
\left(w_{u,i+1}- \bar h_j D_x^{-1}\Big(\sum^k_{m=k-i}\gamma^j_{i,m}L_{j,m}\Big),v\right)_j=0, ~~\forall v\in P^{k-1}(I_j).
\end{align*}
In order to get the superconvergent bounds of $w_{u,i+1}$, we need to find out the expression of coefficient $\beta_{i+1,m}$. After a direct calculation, there hold
\begin{align*}
&\beta^j_{i+1,k-i-1}=-\frac{\gamma^j_{i,k-i}\bar{h}_j}{2(k-i)+1},\ \ \beta^j_{i+1,k-i}=-\frac{\gamma^j_{i,k-i+1}\bar{h}_j}{2(k-i)+3},\\
&\beta^j_{i+1,m}=\bar{h}_j\Big(\frac{\gamma^j_{i,m-1}}{2m-1}-
\frac{\gamma^j_{i,m+1}}{2m+3}\Big),\ \ m=k-i+1,\cdots,k-1.
\end{align*}
Moreover, by the fact that $w^{(\theta)}_{u,i+1}=0$, we get
\begin{align*}
\theta(\beta^j_{i+1,k-i-1}+\cdots+\beta^j_{i+1,k})+
(1-\theta)(-1)^{k-i-1}\beta^{j+1}_{i+1,k-i-1}+\cdots
+(1-\theta)(-1)^k\beta^{j+1}_{i+1,k}=0.
\end{align*}
Again, we can write the above linear system into the matrix-vector form
\begin{align*}
A\beta_{i+1,k}=c,
\end{align*}
and when $\theta\neq\frac{1}{2}$, we arrive at the unique existence of the system. Consequently, it follows from Lemma \ref{27} that
\begin{align*}
\|\partial_tw_{u,i+1}\|_{\infty,I_j}
&\lesssim \sum^k_{m=k-i-1}|\partial_t\beta^j_{i+1,m}|\lesssim h\sum^k_{m=k-i}|\partial_t\gamma^j_{i,m}|\\
&\lesssim h\|\partial_tw_{q,i}\|_{\infty}\lesssim h^{k+i+2}\|\partial_tq\|_{k+i+1,\infty}\lesssim h^{k+i+2}\|u\|_{k+i+4,\infty}.
\end{align*}
 Analogously, the other estimate of \eqref{6} can be obtained, and the orthogonality property in \eqref{7} is a trivial consequence of expression of $w_{u,i}$ and $w_{q,i}$ in \eqref{17} with $i$ replaced by $i+1$. This finishes the proof of Lemma \ref{w}.
\end{proof}

We are now ready to define the correction functions as follows. For any positive integer $\ell\in \mathbb{Z}_k$, we define in each element $I_j$
\begin{equation} \label{8}
W^\ell_u=\sum^\ell_{i=1}w_{u,i},\ \ \ \ \ W^\ell_q=\sum^\ell_{i=1}w_{q,i},
\end{equation}
and the special interpolation functions are
\begin{equation} \label{9}
u^\ell_I=P_\theta u-W^\ell_u,\ \ \ \ \ q^\ell_I=P_{\tilde{\theta}} q-W^\ell_q.
\end{equation}

\begin{lemma}\label{ww}
Suppose $u\in W^{k+\ell+3,\infty}(\Omega), \ell\in \mathbb{Z}_k$ is the solution of \eqref{1}, and $u^\ell_I, q^\ell_I$ are defined by (\ref{9}), then for $\forall v, \psi \in V^k_h$, we have

\begin{subequations}\label{10}
\begin{equation}
|\big((u-u^\ell_I)_t,v\big)_j-(W^\ell_u,v_x)_j+(W^\ell_q,v_x)_j|\lesssim h^{k+\ell+1}\|u\|_{k+\ell+3,\infty}\|v\|_{1, I_j},
\end{equation}
\begin{equation}
~~~~~~~~~~~~~~~~~~~~|(q-q^\ell_I,\psi)_j+(W^\ell_u,\psi_x)_j|\lesssim h^{k+\ell+1}\|u\|_{k+\ell+2,\infty}\|\psi\|_{1, I_j}\label{11}.
\end{equation}
\end{subequations}
\end{lemma}

\begin{proof}\quad By the orthogonality property of $w_{u,i}\ \text{and}\ w_{q,i}, i\in\mathbb{Z}_{k-1}$, we have
\begin{align*}
&D^{-1}_xw_{u,i}(x^-_{j+\frac{1}{2}})=\frac{1}{\bar{h}_j}(w_{u,i},1)_j=0=D^{-1}_xw_{u,i}(x^+_{j-\frac{1}{2}}),\\
&D^{-1}_xw_{q,i}(x^-_{j+\frac{1}{2}})=\frac{1}{\bar{h}_j}(w_{q,i},1)_j=0=D^{-1}_xw_{q,i}(x^+_{j-\frac{1}{2}}).
\end{align*}
It follows from integration by parts that
\begin{align*}
&(\partial_tw_{u,i},v)_j=-\bar{h}_j(D^{-1}_x\partial_tw_{u,i},v_x)_j=-(w_{q,i+1}-w_{u,i+1},v_x),\ v\in V^k_h,\ i\in\mathbb{Z}_{k-1},\\
&(w_{q,i},v)_j=-\bar{h}_j(D^{-1}_xw_{q,i},v_x)_j=-(w_{u,i+1},v_x),\ \ \ \ \ \ \ \ ~~~~~~~~~~v\in V^k_h,\ i\in\mathbb{Z}_{k-1}.
\end{align*}
Then
\begin{align*}
\ &\big((u-u^\ell_I)_t,v\big)_j-(W^\ell_u,v_x)_j+(W^\ell_q,v_x)_j\\
=&\big((u-P_{\theta}u)_t,v\big)_j+\sum^\ell_{i=1}[(\partial_tw_{u,i},v)_j+(w_{q,i}-w_{u,i},v_x)_j]\\
=&(\partial_tw_{u,\ell},v)_j.
\end{align*}
Similarly, there holds
\begin{align*}
(q-q^\ell_I,\psi)_j+(W^\ell_u,\psi_x)_j=(w_{q,\ell},\psi)_j.
\end{align*}
By \eqref{6}, we can get the desired result (\ref{10}).
\end{proof}

\textcolor{red}{\subsection{The case with $\lambda \ne \theta$ in \eqref{23}}}

When \textcolor{red}{parameters $\lambda$ and $\theta$}  in \eqref{23} pertaining to convection and diffusion terms are chosen differently,  a pair of suitable interpolation functions in possession of supercloseness property should be constructed, which are based on a combination of modified projections and new correction functions. To do that, let us first recall a new modified projection\cite{Cheng1}. That is,
\begin{align*}
\Pi_h(u,q)=(P_\theta u, P^*_{\tilde{\theta}}q),
\end{align*}
in which $P_\theta u\in V^k_h$ has been given in \eqref{18}, and $P^*_{\tilde{\theta}}q\in V^k_h$ depends on both $u$ and $q$ satisfying
\begin{align*}
&\int_{I_j}(P^*_{\tilde{\theta}}q)v_h dx=\int_{I_j}qv_h dx,\ \ \ \ \forall v_h\in P^{k-1}(I_j),\\
&(P^*_{\tilde{\theta}}q)^{(\tilde{\theta})}_{j+\frac{1}{2}}
=(q^{(\tilde{\theta})})_{j+\frac{1}{2}}+(\lambda-\theta)[u-P_\theta u]_{j+\frac{1}{2}}
\end{align*}
for any $j=1,\cdots,N$. Moreover, this projection have the following approximation property
\begin{align*}
\|q-P^*_{\tilde{\theta}}q\|_{I_j}\leq Ch^{k+\frac{3}{2}}\Big(\|q\|_{k+1,\infty}+|\lambda-\theta|
\cdot\|u\|_{k+1,\infty}\Big).
\end{align*}
From the above estimate of $q-P^*_{\tilde{\theta}}q$, it is easy to see that the coefficient $\bar{q}_{j.k}$ can be controlled by the prime and auxiliary variables, namely
\begin{align*}
|\bar{q}_{j,k}|&\lesssim \frac{2k+1}{h_j}|(q-P^*_{\tilde{\theta}}q,L_{j,k})|\\
&\lesssim h^{k+1}\Big(\|q\|_{k+1,\infty}+|\lambda-\theta|\cdot\|u\|_{k+1,\infty}\Big)\\
& \lesssim h^{k+1}\|u\|_{k+2,\infty}.
\end{align*}

Next, the corresponding correction functions pertaining to two different weights $\lambda$ and $\theta$ can be easily defined. Specifically,
we define the functions $w_{u,i}, w_{q,i}, i\in\mathbb{Z}_k$ satisfying
\begin{subequations}\label{29}
\begin{align}
&(w_{u,i}-\bar{h}_jD_x^{-1}w_{q,i-1},z)_j=0,\ \ \ \ \ \ \ \ \ \ \ \ \ \ \big(w^{(\theta)}_{u,i}\big)_{{j+\frac{1}{2}}}=0,\\
&(w_{q,i}-w_{u,i}-\bar{h}_jD_x^{-1}\partial_tw_{u,i-1},z)_j=0,\ \ \ \big(w^{(\tilde{\theta})}_{q,i}\big)_{{j+\frac{1}{2}}}=\big(w^{(\lambda)}_{u,i}\big)_{{j+\frac{1}{2}}},\label{31}
\end{align}
\end{subequations}
where $z\in P^{k-1}(I_j)$, and $
w_{u,0}=u-P_\theta u, w_{q,0}=q-P^*_{\tilde{\theta}} q.
$

\textcolor{green}
{
Let us finish this section by providing the following theorem.
\begin{theorem}\label{30}
When $\lambda\neq \theta$ in \eqref{23}, the functions $w_{u,i}, w_{q,i}, i\in\mathbb{Z}_k$ defined in (\ref{29}) still have the following properties
\begin{align*}
&\|\partial_tw_{u,i}\|_{\infty}\lesssim h^{k+i+1}\|u\|_{k+i+3,\infty},~~~ \|w_{q,i}\|_{\infty}\lesssim h^{k+i+1}\|u\|_{k+i+2,\infty}, \\
&(w_{u,i},v)_j=0,\ \qquad (w_{q,i},v)_j=0,\ \qquad \forall v\in P^{k-i-1}(I_j).
\end{align*}
Moreover, when $u\in W^{k+\ell+3,\infty}(\Omega), \ell\in \mathbb{Z}_k$, the special interpolation functions
$$
u^\ell_I=P_\theta u-W^\ell_u,\ \ \ \ \ q^\ell_I=P^*_{\tilde{\theta}} q-W^\ell_q
$$
with \eqref{8} satisfy
\begin{align*}
|\big((u-u^\ell_I)_t,v\big)_j-(W^\ell_u,v_x)_j+(W^\ell_q,v_x)_j|&\lesssim h^{k+\ell+1}\|u\|_{k+\ell+3,\infty}\|v\|_{1, I_j}, \\
|(q-q^\ell_I,\psi)_j+(W^\ell_u,\psi_x)_j| & \lesssim h^{k+\ell+1}\|u\|_{k+\ell+2,\infty}\|\psi\|_{1, I_j}.
\end{align*}
\end{theorem}
\begin{proof}
Since there is only slight difference between \eqref{28} and \eqref{29} in terms of different boundary collocations, Theorem \ref{30} can thus be proved by an argument similar to that in subsection \ref{same} with different vectors $b$, $\tilde b$ and $c$, etc. The detailed proof is omitted.
\end{proof}
}

\section{Superconvergence}\label{super}

In this section, we will show the superconvergence properties for the LDG solution at some special points as well as cell averages, which are mainly based on the supercloseness result for the error between the LDG solution $(u_h,q_h)$ and newly designed interpolation functions $(u^\ell_I,q^\ell_I)$.
\begin{theorem}\label{superclose}
Let $u\in W^{k+\ell+3,\infty}(\Omega)$, $\ell\in \mathbb{Z}_k$ is the exact solution of \eqref{1}, and $u_h, q_h$ are the numerical solutions of LDG scheme (\ref{2}), Then for periodic boundary conditions
\begin{align*}
\|u^\ell_I-u_h\|+\Big(\int^t_0\|q^\ell_I-q_h\|^2d\tau\Big)^{\frac{1}{2}}\le C (1+t)h^{k+\ell+1},
\end{align*}
where $C$ depends on $\|u\|_{k+\ell+3,\infty}$.
\end{theorem}

\begin{proof}
Using Lemma \ref{ww}, we obtain
\begin{align*}
&\quad \!~\big|(\eta_{ut},v)_j+(\eta_q,\psi)_{j}+\mathcal{H}^1_j(\eta_u,\eta_q;v)
+\textcolor{green}{\mathcal{H}^2_j(\eta_u;\psi)}\big|\\
&=\big|\big((u-u^\ell_I)_t,v\big)_j-(W^\ell_u,v_x)_j+(W^\ell_q,v_x)_j
+(q-q^\ell_I,\psi)_j+(W^\ell_u,\psi_x)_j\big|\\
&=\big|(\partial_tw_{u,\ell},v)_j+(w_{q,\ell},\psi)_j|\\
&\lesssim  h^{k+\ell+1}\|u\|_{k+\ell+3,\infty}(\|v\|_{1, I_j}+\|\psi\|_{1, I_j}).
\end{align*}
Inserting the above estimate into \eqref{24} and summing over all $j$, one has
\begin{align*}
\frac{1}{2}\frac{d}{dt}\|\xi_u\|^2+\|\xi_q\|^2
\lesssim  h^{k+\ell+1}\|u\|_{k+\ell+3,\infty}(\|\xi_u\|+\|\xi_q\|).
\end{align*}
If we choose a suitable initial condition such that $$\|\xi_u(0)\|=0,$$ then Theorem \ref{superclose} follows by using Young's inequality and Gronwall inequality.
\end{proof}

\subsection{Superconvergence of numerical fluxes}
In this subsection, we present the superconvergence results of the numerical fluxes.

\begin{theorem}\label{nt}
Assume that $u\in W^{2k+3,\infty}(\Omega), k\geq 1$ is the solution of \eqref{1}, and the $u_h, q_h$ are the numerical solutions of LDG scheme (\ref{2}) with the initial solution $u_h(\cdot,0)=u^k_I(\cdot,0)$. Then for periodic boundary conditions
$$
\|e_{u,n}\| \lesssim C (1+t)h^{2k+1}, \ \ \ \ \  \textcolor{red}{\Big(\int^t_0\|e_{q,n}\|^2d\tau\Big)^{\frac{1}{2}} \lesssim C (1+t)h^{2k+1},}
$$
where
\begin{align*}
\|e_{v,n}\|=\Big(\frac{1}{N}\sum^N_{j=1}
\big|(v-\hat{v}_h)(x_{j+\frac{1}{2}},t)\big|^2\Big)^\frac{1}{2}, \ \ \ v=u,q.
\end{align*}
%\begin{align*}
%\|e_{u,n}\|=\Big(\frac{1}{N}\sum^N_{j=1}
%\big|(u-\hat{u}_h)(x_{j+\frac{1}{2}},t)\big|^2\Big)^\frac{1}{2}.
%\end{align*}
\end{theorem}
\begin{proof}
It follows from the inverse inequality and the supercloseness result in Theorem \ref{superclose} that
\begin{align*}
\|e_{u,n}\|
&=\Big(\frac{1}{N}\sum^N_{j=1}\big|(\hat{u}_I-\hat{u}_h)(x_{j+\frac{1}{2}},t)
\big|^2\Big)^\frac{1}{2}\\
&\lesssim \Big(\frac{1}{N}\sum^N_{j=1}h_j^{-1}\|u_I-u_h\|_{I_j \cup I_{j+1}}^2\Big)^\frac{1}{2} \\
& \lesssim \|u^k_I-u_h\| \lesssim  C (1+t)h^{2k+1}.
\end{align*}
\textcolor{red}{Using the supercloseness result in Theorem \ref{superclose} again, superconvergence of the auxiliary variable $q$ can be derived analogously.} This finishes the proof of Theorem \ref{nt}.
\end{proof}
\subsection{Superconvergence for cell averages}
\begin{theorem}
Assume that the conditions of Theorem \ref{superclose} are satisfied, then for the periodic boundary conditions, we have
\begin{equation}\label{cell}
\|e_u\|_c\lesssim (1+t)h^{2k+1}\|u\|_{2k+3,\infty},\ \ \ \ \textcolor{red}{\Big(\int^t_0\|e_q\|^2_cd\tau\Big)^{\frac{1}{2}}\lesssim (1+t)h^{2k+1}\|u\|_{2k+3,\infty},}
\end{equation}
where
$$
\|e_v\|_c=\Big(\frac{1}{N}\sum^N_{j=1}
\big(\frac{1}{h_j}(e_v,1)_j\big)^2\Big)^\frac{1}{2},\ \ \ v=u,q.
$$
%$$
%\|e_u\|_c=\Big(\frac{1}{N}\sum^N_{j=1}
%\big(\frac{1}{h_j}(e_u,1)_j\big)^2\Big)^\frac{1}{2}.
%$$
\end{theorem}

\begin{proof}
\textcolor{red}{ Taking $\|e_u\|_c$ as an example,}
by the properties of $P_\theta$ and the definition of $u^k_I$, we obtain
\begin{equation}
(e_u, 1)_j = (u^k_I-u_h, 1)_j+\ (W^k_u, 1)_j .\label{12}
\end{equation}
The superconvergent result can thus be proved by using the orthogonality property in \eqref{7}, Cauchy--Schwarz inequality and Theorem \ref{superclose}.
\end{proof}

\subsection{Superconvergence at  generalized Radau points}

As a natural extension of Radau points for $\theta = 1$, the roots of generalized Radau polynomials for the weight $\theta$ are introduced in \cite{Ryan1}. To be more specific, the generalized Radau polynomials are defined as
\begin{equation}\label{0920}
R^*_{k+1}=\begin{cases}
L_{k+1}-(2\theta-1)L_k,\ \ \ \ \ \text{when \ $k$ \ is \ even},\\
(2\theta-1)L_{k+1}-L_k,\ \ \ \ \ \text{when \ $k$ \ is \ odd}.
\end{cases}
\end{equation}

For superconvergence analysis, instead of using the global projection $P_\theta u$, a much simpler local projection $P_hu$ is introduced \cite{Cao6}:
\begin{align*}
&\int_{I_j}(P_hu-u)v=0,\ \ \ \forall v\in P^{k-1}(I_j), \\
&\theta P_hu(x^-_{j+\frac{1}{2}})+(1-\theta)P_hu(x^+_{j-\frac{1}{2}})=\theta u(x^-_{j+\frac{1}{2}})+(1-\theta)u(x^+_{j-\frac{1}{2}}).
\end{align*}

\begin{lemma}\cite{Cao6}
Suppose $u\in W^{k+2,\infty}(\Omega)$ and $P_hu$ is the local projection of $u$ defined above with $\theta\neq\frac{1}{2}$, then
\begin{align*}
&|(u-P_hu)(R^r_{j,m})|\lesssim h^{k+2}\|u\|_{k+2,\infty},\ \ \\
&|\partial_x(u-P_hu)(R^{r*}_{j,m})|\lesssim h^{k+1}\|u\|_{k+2,\infty},\\
&\|P_hu-P_\theta u\|_{\infty}\lesssim h^{k+2}\|u\|_{k+2,\infty}.
\end{align*}
Here $R^r_{j,m}, R^{r*}_{j,m}$ are the roots of rescaled Radau polynomials $R^*_{j,m+1}$ and $\partial_xR^*_{j,m+1}$.
\end{lemma}
We are now ready to show the superconvergence result at generalized Radau points.
\begin{theorem}\label{radau}
Let $u\in W^{k+5,\infty}(\Omega)$ and $u_h$ be the numerical solution of \eqref{1}, suppose $u^\ell_I, \ell\geq 2$ is the special interpolation function defined in (\ref{9}), then for the periodic boundary conditions
\begin{align*}
& ~ \|e_{u,r}\|\lesssim (1+t)h^{k+2}\|u\|_{k+5,\infty},\ \ \ \ \ \ \ \ \ \ \ \ \ \|e_{u,rx}\|\lesssim (1+t)h^{k+1}\|u\|_{k+5,\infty},\\
& \textcolor{red}{\Big(\int^t_0\|e_{q,l}\|^2d\tau\Big)^\frac{1}{2} \lesssim (1+t)h^{k+2}\|u\|_{k+5,\infty},\ \Big(\int^t_0\|e_{q,lx}\|^2d\tau\Big)^\frac{1}{2}
     \lesssim (1+t)h^{k+1}\|u\|_{k+5,\infty},}
\end{align*}
where
\begin{align*}
& \|e_{u,r}\|=\max\limits_{j\in\mathbb{Z}_N}|(u-u_h)(R^r_{j,m})|, \ \ \ \  \|e_{u,rx}\|=\max\limits_{j\in\mathbb{Z}_N}|(u-u_h)_x(R^{r*}_{j,m})|,\\
& \textcolor{red}{\|e_{q,l}\|=\max\limits_{j\in\mathbb{Z}_N}|(q-q_h)(R^l_{j,m})|, \ \ \ \ \ \|e_{q,lx}\|=\max\limits_{j\in\mathbb{Z}_N}|(q-q_h)_x(R^{l*}_{j,m})|.}
\end{align*}
\textcolor{red}{Here, $R^l_{j,m}$ and $R^{l*}_{j,m}$ are the roots of rescaled Radau polynomials $R^*_{j,m+1}$ and $\partial_xR^*_{j,m+1}$ in \eqref{0920} with $\theta$ replaced by $\tilde \theta$.}
\end{theorem}

\begin{proof}
By choosing $\ell=2$ in Theorem \ref{superclose}, we obtain
\begin{align*}
\|u_h-u^2_I\|\lesssim (1+t)h^{k+3}\|u\|_{k+5,\infty}.
\end{align*}
From the inverse inequality, we can get
\begin{align*}
\|u_h-u^2_I\|_{\infty}\lesssim h^{-\frac{1}{2}}\|u_h-u^2_I\|\lesssim (1+t)h^{k+\frac{5}{2}}\|u\|_{k+5,\infty}.
\end{align*}
By the the triangle inequality,
\begin{align*}
|(u-u_h)(R^r_{j,m})|
&\lesssim\|u_h-u^2_I\|_{\infty}+|(u-P_hu)(R^r_{j,m})|+\|P_hu-P_\theta u\|_{\infty}+\|W^2_u\|_\infty\\
&\lesssim h^{k+2}\|u\|_{k+2,\infty}.
\end{align*}
\textcolor{red}{The superconvergence results for the derivative of errors and the auxiliary variable $q$ can be obtained by the same arguments. This completes the proof of Theorem \ref{radau}.}
% The results of the other equation for the derivative of errors can be obtained by the same arguments.
\end{proof}

\textcolor{red}
{\begin{remark}
The analysis of superconvergence is mainly based on the supercloseness between the LDG solution $(u_h,q_h)$ and the interpolation function $(u^\ell_I,q^\ell_I)$ by asking for $(W^\ell_u,W^\ell_q)$ satisfying
$$
\big(W^\ell_u\big)^{(\theta)}_{j+\frac{1}{2}}=0, \ \ \ \
\big(W^\ell_q\big)^{(\tilde{\theta})}_{j+\frac{1}{2}}=0,\ \ \ j\in\mathbb{Z}_N.
$$
Therefore, when $\lambda \neq \theta$, the superconvergence results for auxiliary variable $q$ are no longer valid, since \eqref{31} is needed  indicating that $\big(W^\ell_q\big)^{(\tilde{\theta})}_{j+\frac{1}{2}}\ne 0$. In addition,  when $\lambda = \theta$, superconvergence of $q$ can be proved in the $L^2([0, T]; L^2[0, 2 \pi])$ norm while superconvergence can be observed numerically in the $L^\infty([0, T]; L^2[0, 2 \pi])$ norm.
\end{remark}
}
\section{\textcolor{green}{Other boundary conditions}}\label{extension}
\subsection{\textcolor{green}{Mixed boundary conditions}}
For the mixed boundary conditions
\begin{align}\label{19}
u(0,t)=g_1(t), \ \ \ \ \ u_x(2\pi,t)=g_2(t),
\end{align}
the numerical fluxes are chosen as
\begin{equation}\label{21}
(\tilde{u}_h-\hat{q}_h,\hat{u}_h)_{j+\frac{1}{2}}=\begin{cases}
(g_1-q_h^+,g_1),\ \ \ j=0,\\
(u^{\theta}_h-q^{\tilde{\theta}}_h, u^{\theta}_h),\ \ \ j=1,\cdots,N-1,\\
(u^-_h-g_2,u^-_h),\ \  j=N.
\end{cases}
\end{equation}

The corresponding global projections $P_\theta$ and  $P_{\tilde \theta}$ are modified to be in the following piecewise global version, i.e.,
\begin{equation*}
\begin{cases}
(\tilde{P}_\theta u,v)_j=(u,v)_j,\ \forall v\in P^{k-1}(I_j),\\
(\tilde{P}_\theta u)^{(\theta)}_{j+\frac{1}{2}}=u^{(\theta)}_{j+\frac{1}{2}},~~ j=1,\cdots,N-1,\\
(\tilde{P}_\theta u)^-_{N+\frac{1}{2}}=u^-_{{N+\frac{1}{2}}},~ j=N,
\end{cases}
\end{equation*}
and
\begin{equation}\label{22}
\begin{cases}
(\tilde{P}_{\tilde{\theta}} q,\eta)_j=(q,\eta)_j,\ \forall \eta\in P^{k-1}(I_j),\\
(\tilde{P}_{\tilde{\theta}} q)^{(\tilde{\theta})}_{j-\frac{1}{2}}=q^{(\tilde{\theta})}_{j-\frac{1}{2}},~~~ j=2,\cdots,N,\\
(\tilde{P}_{\tilde{\theta}} q)^+_{\frac{1}{2}}=q^+_{\frac{1}{2}}, ~~~~~~~~\!j=1.
\end{cases}
\end{equation}
Obviously, the projection $\tilde{P}_\theta$ can be decoupled staring from the cell $I_N$ and $\tilde{P}_{\tilde{\theta}}$ can be computed from the cell $I_1$.
Moreover, we have the following optimal approximation properties.
\begin{lemma}\cite{Meng3}
Assume $z\in W^{k+1,\infty}(I_j)$ with $\theta\neq\frac{1}{2}$, the projection $P=\tilde{P}_\theta$ or $\tilde{P}_{\tilde{\theta}}$ defined above satisfies the following approximation property
\begin{align*}
\|z-P z\|_{I_j}+h^{\frac{1}{2}}\|z-P z\|_{\infty,I_j}\leq Ch^{k+\frac{3}{2}}\|z\|_{k+1,\infty},
\end{align*}
where $C$ is independent of $h$ and $z$.
\end{lemma}

Replacing $P_\theta$ ($P_{\tilde \theta}$) by $\tilde{P}_\theta$ ($\tilde{P}_{\tilde{\theta}}$), we are able to construct the following correction functions in possession of supercloseness properties. That is,
for $z \in P^{k-1}(I_j)$
\begin{align*}
&(w_{u,i}-\bar{h}_jD_x^{-1}w_{q,i-1},z)_j=0,\ \ \ \ \ \ \ \ \ \ \ \ \ \ ~ \big(w^{(\theta)}_{u,i}\big)_{{j+\frac{1}{2}}}=0,\ \forall j\in \mathbb{Z}_{N-1},\\
&(w_{q,i}-w_{u,i}-\bar{h}_jD_x^{-1}\partial_tw_{u,i-1},z)_j=0,\ \ \ \big(w^{(\tilde{\theta})}_{q,i}\big)_{{j+\frac{1}{2}}}=0,\ \forall j\in \mathbb{Z}_{N-1},\\
&\qquad \qquad \qquad \qquad \qquad \qquad \qquad\qquad \qquad \big(w^-_{u,i}\big)_{{N+\frac{1}{2}}}=0,\ \ \big(w^+_{q,i}\big)_{\frac{1}{2}}=0.
\end{align*}
The superconvergent results can thus be obtained if we follow the same argument as that in section \ref{correction} and section \ref{super}.

\subsection{Dirichlet boundary conditions}
For Dirichlet boundary conditions
\begin{align}\label{20}
u(0,t)=g_3(t),\ \ \ \ \ \ u(2\pi,t)=g_4(t),
\end{align}
following \cite{Castillo}, we choose the numerical fluxes as
\begin{equation*}
(\tilde{u}_h-\hat{q}_h,\hat{u}_h)_{j+\frac{1}{2}}=\begin{cases}
(g_3-q_h^+,g_3),\ \ \ \ \ \ j=0,\\
(u^{\theta}_h-q^{\tilde{\theta}}_h, u^{\theta}_h),\ \ ~\ \ \ j=1,\cdots,N-1,\\
(u^-_h-q_h^-,g_4),\ \ \ \ \ j=N.
\end{cases}
\end{equation*}

Similarly, we still need to make slight changes to the projection. For projection $\tilde{P}_{\tilde{\theta}}$, we still adopt the definition in (\ref{22}), while the projection $\tilde{P}_\theta$ is modified as follows:
\begin{equation*}
\begin{cases}
(\tilde{P}_\theta u,v)_j=(u,v)_j,\ \forall v\in P^{k-1}(I_j),\\
(\tilde{P}_\theta u)^{(\theta)}_{j+\frac{1}{2}}=u^{(\theta)}_{j+\frac{1}{2}},\ \ j\in \mathbb{Z}_{N-1},\\
\ (\tilde{P}_\theta u)^-_{N+\frac{1}{2}}=u^-_{N+\frac{1}{2}}+(\tilde{P}_{\tilde{\theta}}q-q)^-_{N+\frac{1}{2}}.
\end{cases}
\end{equation*}
From the last equation we can see that, compared with the mixed boundary condition, the left limit of projection at point $x_{N+\frac{1}{2}}$ consists of two parts. One is the left limit of exact solution $u$ at point $x_{N+\frac{1}{2}}$, the other is the left limit of projection error of the auxiliary variable $q$ at point $x_{N+\frac{1}{2}}$. Since we do not have any information about the auxiliary variable $q$ at the boundary, we need to use the prime variable $u$ to eliminate the boundary term introduced by $\tilde{P}_{\tilde{\theta}}q-q$ at point $x_{N+\frac{1}{2}}$.

The superconvergent results can be obtained if we define the following correction functions:
for $z \in P^{k-1}(I_j)$,
\begin{align*}
&(w_{u,i}-\bar{h}_jD_x^{-1}w_{q,i-1},z)_j=0,\ \ \ \ \ \ \ \ \ \ \ \ \ \ \big(w^{(\theta)}_{u,i}\big)_{{j+\frac{1}{2}}}=0,\ \forall j\in \mathbb{Z}_{N-1},\\
&(w_{q,i}-w_{u,i}-\bar{h}_jD_x^{-1}\partial_tw_{u,i-1},z)_j=0,\ \ \ \big(w^{(\tilde{\theta})}_{q,i}\big)_{{j+\frac{1}{2}}}=0,\ \forall j\in \mathbb{Z}_{N-1},\\
&\qquad \qquad \qquad\qquad\qquad \qquad \qquad\qquad \qquad \big(w^+_{q,i}\big)_{\frac{1}{2}}=0,\ \big(w^-_{u,i}\big)_{{N+\frac{1}{2}}}=\big(w^-_{q,i}\big)_{{N+\frac{1}{2}}}.
\end{align*}

\subsection{Initial discretization}
In this section, we consider how to discretize the initial datum. Initial value discretization is very important for the study of superconvergence, which can be obtained using the same technique as that in \cite{Cao6}. Specifically, for periodic boundary conditions,
\begin{enumerate}
  \item according to the definition of projection $P_\theta, P_{\tilde{\theta}}$, calculate the $w_{u,0}, w_{q,0}$;
  \item calculate $w_{u,i}, w_{q,i}$ by the equations (\ref{28});
  \item calculate $W^\ell_u=\sum\limits^\ell_{i=1}w_{u,i}$, $u^\ell_I=P_\theta u-W^\ell_u$;
  \item let $u_h(\cdot,0)=u^\ell_I(\cdot,0)$.
\end{enumerate}

\section{Numerical results}\label{numerical}
In this section, we provide numerical examples to illustrate our theoretical findings. For time discretization, we use TVDRK3 method and take $\Delta t=CFL*h^2$.
\begin{example}
We consider the following problem
\begin{align*}
&u_t+u_x-u_{xx}=0,\ \ \ \ \ \ \ (x,t)\in [0,2\pi]\times (0,T],\\
&u(x,0)=\sin(x)-x,\ \ \ \ \ x\in R
\end{align*}
with the periodic boundary condition, where the exact solution is
\begin{align*}
u(x,t)=e^{-t}\sin(x-t).
\end{align*}
\end{example}

\begin{table}[!htp]
\caption{\label{samew} Errors and orders for $u$. $\lambda=\theta, T=1.0$, $k=2,3,4$. }
\begin{center}
\small
\scalebox{1.0}{\begin{tabular} {c| c c c c c c c c c} \hline\hline

                 &$N$    &$\|e_{un}\|$  &Order &$\|e_u\|_c$ &Order &$\|e_{u,r}\|$  &Order &$\|e_{u,rx}\|$  &Order \\ \hline
$k=2$            &20   &5.20E-08  &--    &1.91E-07    &--    &4.53E-06  &--    &6.95E-05   &-- \\
$CFL=0.01$       &40   &1.83E-09  &4.83  &6.23E-09    &4.94  &2.80E-07  &4.01  &8.74E-06   &2.99\\
$\lambda=0.8$    &80   &6.09E-11  &4.91  &1.99E-10    &4.96  &1.74E-08  &4.01  &1.10E-06   &2.99\\
$\theta=0.8$     &160  &1.96E-12  &4.96  &6.32E-12    &4.98  &1.08E-09  &4.00  &1.38E-07   &2.99\\\hline

$k=3$            &15   &5.35E-10  &--    &6.62E-10    &--    &1.90E-07  &--    &1.27E-05   &-- \\
$CFL=0.005$      &30   &3.82E-12  &7.13  &5.69E-12    &6.86  &5.53E-09  &5.10  &7.86E-07   &4.02\\
$\lambda=0.9$    &45   &2.07E-13  &7.19  &3.50E-13    &6.88  &7.12E-10  &5.05  &1.55E-07   &4.00\\
$\theta=0.9$     &60   &2.66E-14  &7.13  &4.80E-14    &6.91  &1.67E-10  &5.01  &4.89E-08   &4.01\\\hline

$k=4$            &10   &1.60E-11  &--    &5.08E-11    &--    &8.34E-08  &--    &7.88E-06   &--  \\
$CFL=0.001$      &15   &2.04E-13  &10.76 &1.23E-12    &9.17  &7.35E-09  &5.98  &1.06E-06   &4.94\\
$\lambda=1.2$    &20   &1.07E-14  &10.23 &7.91E-14    &9.54  &1.31E-09  &6.01  &2.54E-07   &4.98\\
$\theta=1.2$     &25   &6.74E-15  &2.10  &9.02E-15    &9.73  &3.41E-10  &6.01  &8.33E-08   &4.99\\\hline\hline

\end{tabular}}
\end{center}
\end{table}

\begin{table}[!htp]
\caption{\label{samewqq} Errors and orders for $q$. $\lambda=\theta, T=1.0$, $k=2,3,4$. }
\begin{center}
\small
\scalebox{1.0}{\begin{tabular} {c| c c c c c c c c c} \hline\hline

                 &$N$    &$\|e_{qn}\|$  &Order &$\|e_q\|_c$ &Order &$\|e_{q,l}\|$  &Order &$\|e_{q,lx}\|$  &Order \\ \hline
$k=2$            &20   &1.28E-07  &--    &2.50E-08    &--    &5.10E-06  &--    &8.50E-05   &-- \\
$CFL=0.01$       &40   &4.15E-09  &4.94  &1.04E-09    &4.59  &3.19E-07  &3.99  &1.06E-05   &3.00\\
$\lambda=0.7$    &80   &1.33E-10  &4.96  &3.76E-11    &4.79  &2.00E-08  &4.00  &1.33E-06   &2.99\\
$\theta=0.7$     &160  &4.21E-12  &4.98  &1.26E-12    &4.90  &1.25E-09  &4.00  &1.67E-07   &3.00\\\hline

$k=3$            &15   &1.55E-09  &--    &5.31E-10    &--    &4.14E-07  &--    &1.44E-05   &-- \\
$CFL=0.005$      &30   &1.26E-11  &6.94  &3.81E-12    &7.12  &1.25E-08  &5.05  &8.92E-07   &4.01\\
$\lambda=0.9$    &45   &7.52E-13  &6.96  &2.07E-13    &7.19  &1.63E-09  &5.02  &1.76E-07   &4.01\\
$\theta=0.9$     &60   &1.02E-13  &6.93  &2.66E-14    &7.12  &3.83E-10  &5.03  &5.54E-08   &4.01\\\hline

$k=4$            &10   &5.72E-11  &--    &1.58E-11    &--    &1.02E-07  &--    &8.12E-06   &--  \\
$CFL=0.001$      &15   &1.60E-12  &8.82  &2.03E-13    &10.74 &8.79E-09  &6.05  &1.11E-06   &4.91\\
$\lambda=1.2$    &20   &1.32E-13  &8.68  &1.08E-14    &10.21 &1.55E-09  &6.03  &2.63E-07   &5.00\\
$\theta=1.2$     &25   &2.00E-14  &8.45  &6.80E-15    &2.06  &4.04E-10  &6.03  &8.70E-08   &4.95\\\hline\hline

\end{tabular}}
\end{center}
\end{table}

Table \ref{samew} lists the results for $u$ with $\lambda=\theta$, from which we observe $(2k+1)$th order superconvergence for numerical traces as well as cell averages, and  that the convergence order of the error as well as its derivative are $k+2$ and $k+1$, respectively. \textcolor{red}{Table \ref{samewqq} shows errors and orders for $q$, demonstrating that our results hold true for the auxiliary variable when $\lambda = \theta$.} Moreover, the results with different weights for $\lambda ,\theta$ are given in Table \ref{diffw}, and similar conclusions can be observed for $u$, indicating that choosing different parameters for convection term and diffusion term does not affect the superconvergence results as far as the prime variable $u$ is concerned.

\begin{table}[!htp]
\caption{\label{diffw} Errors and orders for $u$. $\lambda \neq \theta, T=1.0$, $k=2,3,4$. }
\begin{center}
\small
\scalebox{1.0}{\begin{tabular} {c| c c c c c c c c c} \hline\hline

                 &$N$    &$\|e_{un}\|$  &Order &$\|e_u\|_c$ &Order &$\|e_{u,r}\|$  &Order &$\|e_{u,rx}\|$  &Order \\ \hline
$k=2$            &20   &1.41E-07  &--    &3.09E-07    &--    &4.75E-06  &--    &6.71E-05   &-- \\
$CFL=0.01$       &40   &4.60E-09  &4.93  &9.89E-09    &4.97  &2.91E-07  &4.03  &8.58E-06   &2.97\\
$\lambda=1.2$    &80   &1.47E-10  &4.96  &3.13E-10    &4.98  &1.80E-08  &4.01  &1.09E-06   &2.97\\
$\theta=0.8$     &160  &4.66E-12  &4.98  &9.88E-12    &4.99  &1.12E-09  &4.01  &1.37E-07   &2.99\\\hline

$k=3$            &15   &1.85E-10  &--    &7.44E-10    &--    &2.59E-07  &--    &4.80E-06   &-- \\
$CFL=0.002$      &30   &1.60E-12  &6.85  &5.46E-12    &7.09  &8.03E-09  &5.01  &3.01E-07   &3.99\\
$\lambda=0.9$    &45   &9.64E-14  &6.93  &3.13E-13    &7.04  &1.05E-09  &5.01  &5.96E-08   &3.99\\
$\theta=1.1$     &60   &1.28E-14  &7.02  &4.16E-14    &7.02  &2.50E-10  &5.00  &1.88E-08   &4.00\\\hline

$k=4$            &10   &1.87E-10  &--    &1.69E-10    &--    &8.13E-08  &--    &7.88E-06   &--  \\
$CFL=0.001$      &15   &4.93E-12  &8.97  &4.63E-12    &8.88  &7.20E-09  &5.98  &1.06E-06   &4.95\\
$\lambda=0.8$    &20   &3.54E-13  &9.15  &3.35E-13    &9.12  &1.28E-09  &6.00  &2.53E-07   &4.98\\
$\theta=1.2$     &25   &4.33E-14  &9.41  &4.09E-14    &9.43  &3.38E-10  &5.97  &8.31E-08   &4.99\\\hline\hline

\end{tabular}}
\end{center}
\end{table}

\begin{table}[!htp]\centering
\caption{\label{tablenu} \textcolor{green}{Errors and rates for mixed boundary condition (\ref{19}).}}
\vspace{2.0ex}
{\small
\begin{tabular}[c]{cc|cccc|cccc}\hline\hline
&\multirow{2}*{$N$}
&\multicolumn{4}{c|}{$\lambda=\theta=0.8$}
&\multicolumn{4}{c}{$\lambda=\theta=1.2$}\\ \cline{3-10}
&& $\|e_{un}\|$  &Order  &$\|e_u\|_c$   &Order
&  $\|e_{un}\|$  &Order  &$\|e_u\|_c$   &Order\\ \hline

\multirow{4}*{$P^1$}
&  40  &2.30E-05  &--     &3.48E-05  &--     &8.25E-06  &--     &1.35E-05  &--   \\
&  80  &2.72E-06  &3.08  &4.30E-06  &3.02  &1.06E-06  &2.95  &1.75E-06  &2.94\\
& 160  &3.31E-07  &3.04  &5.35E-07  &3.01  &1.35E-07  &2.97  &2.24E-07  &2.97\\
& 320  &4.07E-08  &3.02  &6.67E-08  &3.00  &1.71E-08  &2.99  &2.82E-08  &2.98\\ \hline

\multirow{4}*{$P^2$}
& 20  &7.36E-08  &--     &1.83E-07  &--     &4.49E-07  &--     &7.09E-07  &--   \\
& 40  &2.05E-09  &5.16  &5.64E-09  &5.02  &1.37E-08  &5.04  &2.24E-08  &4.98\\
& 80  &6.10E-11  &5.07  &1.76E-10  &5.00  &4.16E-10  &5.04  &6.99E-10  &5.00\\
&160  &1.87E-12  &5.04  &5.10E-12  &5.10  &1.27E-11  &5.03  &2.15E-11  &5.02\\ \hline

\multirow{4}*{$P^3$}
& 20  &9.64E-11  &--     &1.56E-10  &--     &3.90E-11  &--     &8.65E-11  &--   \\
& 30  &5.53E-12  &7.05  &9.63E-12  &6.88  &2.26E-12  &7.02  &4.93E-12  &7.07\\
& 40  &7.04E-13  &7.16  &1.32E-12  &6.90  &3.12E-13  &6.88  &6.45E-13  &7.07\\
& 50  &1.65E-13  &6.50  &2.93E-13  &6.77  &9.82E-14  &5.17  &1.82E-13  &5.66\\ \hline\hline
\end{tabular}}
\end{table}

\begin{example}\label{exn}
We consider the following problem
\begin{align*}
&u_t+u_x-u_{xx}=0,\ \ \ \ \ \ \ (x,t)\in [0,2\pi]\times (0,T],\\
&u(x,0)=\sin(x)-x,\ \ \ \ \ \ \ \ \ \ x\in R
\end{align*}
with the mixed boundary conditions
$$u(0,t)=e^{-t}\sin(t)-t,\ \ \ \ \ u_x(2\pi,t)=e^{-t}\cos(t)+1;$$
the exact solution is
\begin{align*}
u(x,t)=e^{-t}\sin(x-t)+x-t.
\end{align*}
\end{example}

The problem is solved by the LDG scheme (\ref{2}) with $k=1,2,3$, respectively, and the numerical fluxes are chosen as (\ref{21}). We list various errors and corresponding convergence rates when $\lambda=\theta=0.8, \lambda=\theta=1.2$ in Table \ref{tablenu}. The superconvergence results of order $2k+1$ at numerical traces as well as cell averages demonstrate that the superconvergence also holds for mixed boundary conditions. In addition, to verify theoretical results for Dirichlet boundary conditions, we consider Example \ref{exn} with the following Dirichlet boundary conditions  $$u(0,t)=e^{-t}\sin(t)+t,\ \ \ \ \ u(2\pi,t)=e^{-t}\cos(t)-1.$$
The results are shown in Table \ref{di}, which confirms that the conclusion still holds for Dirichlet boundary conditions.

\begin{table}[!htp]\centering
\caption{\label{di} Errors and rates for Dirichlet boundary condition (\ref{20}).}
\vspace{2.0ex}
{\small
\begin{tabular}[c]{cc|cccc|cccc}\hline\hline
&\multirow{2}*{$N$}
&\multicolumn{4}{c|}{$\lambda=\theta=0.7$}
&\multicolumn{4}{c}{$\lambda=\theta=0.9$}\\ \cline{3-10}
&& $\|e_{un}\|$  &Order  &$\|e_u\|_c$   &Order
&  $\|e_{un}\|$  &Order  &$\|e_u\|_c$   &Order\\ \hline

\multirow{4}*{$P^1$}
&  40  &3.44E-05  &-     &5.27E-05  &-     &1.54E-05  &-     &2.51E-05  &-   \\
&  80  &4.15E-06  &3.05  &6.62E-06  &2.99  &1.95E-06  &2.98  &3.21E-06  &2.96\\
& 160  &5.09E-07  &3.03  &8.27E-07  &3.00  &2.46E-07  &2.99  &4.07E-07  &2.98\\
& 320  &6.28E-08  &3.02  &1.03E-07  &3.00  &3.09E-08  &2.99  &5.12E-08  &2.99\\ \hline

\multirow{4}*{$P^2$}
& 20  &2.03E-08  &-     &1.13E-07  &-     &7.07E-08  &-     &2.47E-07  &-   \\
& 40  &7.83E-10  &4.69  &3.69E-09  &4.93  &2.34E-09  &4.92  &7.89E-09  &4.97\\
& 80  &3.13E-11  &4.65  &1.20E-10  &4.94  &7.56E-11  &4.95  &2.50E-10  &4.98\\
&160  &1.06E-12  &4.88  &3.81E-12  &4.98  &2.38E-12  &4.99  &7.83E-12  &5.00\\ \hline

\multirow{4}*{$P^3$}
& 20  &1.26E-10  &-     &2.05E-10  &-     &6.23E-11  &-     &1.25E-10  &-   \\
& 30  &7.77E-12  &6.87  &1.17E-11  &7.06  &3.78E-12  &6.91  &6.87E-12  &7.15\\
& 40  &1.05E-12  &6.94  &1.65E-12  &6.81  &5.36E-13  &6.79  &1.02E-12  &6.62\\
& 50  &2.57E-13  &6.33  &3.26E-13  &7.26  &1.98E-13  &4.47  &1.93E-13  &7.48\\ \hline\hline
\end{tabular}}
\end{table}

\section{Concluding Remarks}\label{summary}

In this paper, we obtain the superconvergence of the convection-diffusion equations based on the generalized alternating numerical fluxes. The main techniques are the construction of correction functions and analysis of the generalized Gauss--Radau projections and their modified versions. Different boundary conditions including periodic, mixed and Dirichlet boundary conditions are considered. The sharpness of the theoretical results is confirmed by numerical experiments.
In further work, we will consider the degenerate diffusion problems and multidimensional equations.

\Acknowledgements{This work was supported by National Natural Science Foundation of China (Grants No. 11971132, U1637208, 71773024, 51605114, 11501149) and  the National Key Research and Development Program of China (Grant No. 2017YFB1401801).}


\end{document}